\documentclass[journal]{IEEEtran}

\ifCLASSINFOpdf
   \usepackage[pdftex]{graphicx}
   \graphicspath{{../pdf/}{../jpeg/}}
   \DeclareGraphicsExtensions{.pdf,.jpeg,.png}
\else
   \usepackage[dvips]{graphicx}
  \graphicspath{{../eps/}}
  \DeclareGraphicsExtensions{.eps}
\fi

\usepackage{amsmath}
\usepackage{amsmath, bm}
\usepackage{amsfonts}
\usepackage{amsthm}
\usepackage{url}

\usepackage{algorithm}
\usepackage{algorithmic} 
\usepackage{enumerate}

\usepackage{siunitx}

\theoremstyle{definition}

\newtheorem{theorem}{Theorem}
\theoremstyle{remark}

\usepackage{multirow}
\usepackage{booktabs}
\usepackage{makecell}
\usepackage{threeparttable}
\usepackage{array}
\usepackage{footnote}

\usepackage{xcolor}

\begin{document}

\title{Non-Iterative Solution for Coordinated Optimal Dispatch via Equivalent Projection---Part II: Method and Applications}





\author{Zhenfei~Tan,~\IEEEmembership{Member,~IEEE,}
      Zheng~Yan,~\IEEEmembership{Senior~Member,~IEEE,}
      Haiwang Zhong,~\IEEEmembership{Senior~Member,~IEEE,}
      and~Qing~Xia,~\IEEEmembership{Senior~Member,~IEEE}
      \vspace{-2em}
\thanks{Z. Tan and Z. Yan are with the Key Laboratory of Control of Power Transmission and Conversion (Shanghai Jiao Tong University), Ministry of Education, Shanghai 200240, China. H. Zhong and Q. Xia are with the State Key Laboratory of Power Systems, the Department of Electrical Engineering, Tsinghua University, Beijing 100084, China. Corresponding Author: Z. Yan (e-mail: yanz@sjtu.edu.cn).}

}



%



\maketitle

\begin{abstract}
  This two-part paper develops a non-iterative coordinated optimal dispatch framework, i.e., free of iterative information exchange, via the innovation of the equivalent projection (EP) theory. The EP eliminates internal variables from technical and economic operation constraints of the subsystem and obtains an equivalent model with reduced scale, which is the key to the non-iterative coordinated optimization. In Part II of this paper, a novel projection algorithm with the explicit error guarantee measured by the Hausdorff distance is proposed, which characterizes the EP model by the convex hull of its vertices. This algorithm is proven to yield a conservative approximation within the prespecified error tolerance and can obtain the exact EP model if the error tolerance is set to zero, which provides flexibility to balance the computation accuracy and effort. Applications of the EP-based coordinated dispatch are demonstrated based on the multi-area coordination and transmission-distribution coordination. Case studies with a wide range of system scales verify the superiority of the proposed projection algorithm in terms of computational efficiency and scalability, and validate the effectiveness of the EP-based coordinated dispatch in comparison with the joint optimization.
\end{abstract}



%
\IEEEpeerreviewmaketitle

\bstctlcite{IEEEexample:BSTcontrol} 
\section{Introduction}\label{sec:intro}
\subsection{Background and Literature Review}
\IEEEPARstart{T}{he} coordinated optimal dispatch (COD) of power systems in different regions, voltage levels, and communities is vital for improving the operational economy and security. Conventional coordinated optimization methods rely on iterative information exchange among subsystems, which results in drawbacks including the convergence issue, communication burden, scalability issue, and incompatibility with the serial coordination scheme in practice. In this regard, realizing the COD in a non-iterative fashion, i.e., without iterative information exchange among subsystems, is of great interest to both academia and industry. This two-part paper aims at this issue and proposes the equivalent projection (EP)-based solution. The basic idea is to make external equivalence of technical and economic features of the subsystem for the upper-level coordinated optimization, which requires only a single-round exchange of some boundary information. Following the theory and framework introduced in Part I, the calculation method and typical applications of the EP will be discussed in this present paper.


The EP is mathematically a geometric projection problem, which eliminates internal decision variables from the secure and economic operation constraints of the subsystem and yields a low-dimensional region regarding coordination variables to represent the system. Projection calculation is a challenging task even for linear systems due to the exponential complexity in worst-case scenarios. The most classic projection algorithm is the Fourier-Motzkin Elimination (FME). The FME eliminates variables one by one through linear combinations of inequalities with opposite signs before the coefficient of each variable to be eliminated \cite{ref:fme0}. However, numerous redundant constraints will be generated after eliminating each variable, which not only complicates the projection result, but also increases the number of constraints needed to be processed for eliminating the next variable and thus aggravates the computation burden dramatically \cite{ref:loadability1}. To over these drawbacks, existing studies propose the redundancy identification technique to accelerate the FME calculation and obtain the minimal representation of the projection result \cite{ref:ifme, ref:loadability2}. Another basic algorithm for polyhedral projection is the block elimination \cite{ref:block}, which identifies facets of the projected polytope based on the projection lemma. Though one-by-one elimination of variables is avoided, the block elimination may also generate redundant constraints. The polyhedral projection problem is theoretically proven to be equivalent to the multi-parametric programming (MPP) problem \cite{ref:mpp}, which enables the solution of the projection problem via MPP algorithms. This method also leverages the projection lemma to identify irredundant inequalities of the projection.


The application of above projection methods to the EP calculation in the power system is limited due to features of the coordinated dispatch problem. First, these methods have exponential complexity to the number of variables to be eliminated and thus, are inefficient for power system dispatch problems where the number of internal variables is large. Second, a conservative approximation of the exact EP model is usually desired in engineering applications to reduce the computation effort. However, the aforementioned methods cannot yield an inner approximation since they characterize the projection by generating inequalities and will overestimate the projection result when terminated prematurely without finding all the inequalities. In this regard, existing studies in the power system literature propose different models to approximate the projection region, e.g., the box model \cite{ref:box}, the Zonotope model \cite{ref:zono}, the ellipse model \cite{ref:ellipse}, and the robust optimization-based model \cite{ref:robust}. These methods approximate the projection region with simple shapes at the cost of accuracy. As the approximation shape is pre-fixed, these methods are only applicable to projection regions with specific geometric structures and the approximation accuracy cannot be adjusted flexibly.

Another critical issue regarding the EP calculation is the error metric for controlling the calculation accuracy of the algorithm. The error metric for the EP model measures the difference between the approximated region and the exact one. Reference \cite{ref:zono} uses the distances between parallel facets of the region to measure the approximation quality. However, this metric can only be applied to a certain kind of polytope namely the Zonotope. The volume is a more general metric for comparing two regions. Reference \cite{ref:pq1} measures the calculation accuracy of the active and reactive flexibility region of the distribution network by comparing areas of regions. Reference \cite{ref:pq1} uses the same metric to measure the flexibility region enlargement brought by power flow routers. Based on the volume metric, the Jaccard similarity can be used to measure the relative error of the approximation \cite{ref:jaccard}. However, the evaluation of volume metrics is time-consuming, especially for high-dimensional problems. Furthermore, The evaluation of the volume-based error metric also relies on the knowledge of the exact projected region. Hence, this metric can only be used to test the approximation quality when the ground truth of the projection is known, but can hardly be used to control the approximation accuracy in the projection algorithm.

\subsection{Contributions and Paper Organization}
To meet practical requirements of the EP calculation for power system applications, the Part II of this paper proposes a novel projection algorithm with an explicit accuracy guarantee. With this algorithm, the EP theory and the non-iterative COD method are instantiated for multi-area system coordination and transmission-distribution coordination. The contribution of the Part II is threefold,

1) The progressive vertex enumeration (PVE) algorithm is proposed for the EP calculation in power system applications. Compared to existing methods that solve the projection problem by eliminating internal variables, the proposed method directly identifies the projection region in the lower-dimensional coordination space and thus, has lower computational complexity. Compared to existing approximation methods, the proposed algorithm is proven to be accurate for the polyhedral projection problem. The proposed method can also be used to get an inner approximation with an explicit and adjustable error tolerance according to practical requirements.

2) The Hausdorff distance is employed to measure the approximation error of the PVE algorithm. Compared to the volume metric and the Jaccard similarity, the proposed error metric is a byproduct of the vertex identification in each step of the algorithm and thus, is computationally efficient and can be used to balance the accuracy and computation effort of the algorithm.

3) The non-iterative COD for the multi-area system and transmission-distribution system is realized based on the EP calculation, which thoroughly overcomes the disadvantages of conventional coordinated optimization methods brought by iterations. The EP-based COD is verified to yield identical solutions as the joint optimization and consumes less computation time in scenarios with numerous subsystems.

The remainder of this paper is organized as follows. Section II introduces the PVE algorithm and its properties. Sections III and IV apply the EP-based COD method to the multi-area coordinated dispatch problem and transmission-distribution coordination problem, respectively. Section V concludes this paper and discusses future research.

\section{Calculation Method for EP}\label{sec:algorithm}
\subsection{Problem Setup}
In power system applications, the linear model is mostly utilized for optimal dispatch for the sake of computational efficiency, reliability, and clear economic interpretation. A broad spectrum of literature has investigated the linearized modeling of the transmission network \cite{ref:md_dc, ref:md_liudd}, distribution network \cite{ref:distflow2, ref:dist_wangyi}, and multi-energy system \cite{ref:md_hub}. Nonlinear models, in contrast, have intractable computation and may lead to incentive issues when non-convexity exists \cite{ref:nonconv_price1}. The application of nonlinear optimization is difficult even for centralized dispatch, let alone coordinated dispatch. Part II of this paper focuses on the practice-oriented coordinated dispatch method and thus, the widely used linear dispatch model is considered.

As introduced in Part I, the EP eliminates internal variables from the operation feasible region of the subsystem and makes external equivalence of the subsystem model. The operation feasible region of the subsystem is enforced by both technical constraints and the epigraph of the objective function. Without loss of generality, represent the operation feasible region of the subsystem in the following compact form,
\begin{equation}
  \label{md:ofr}
  \Omega := \left\{(x,y) \in \mathbb{R}^{N_x} \times \mathbb{R}^{N_y} : Ax + By \leq c \right\}.
\end{equation}
In the model, $x$ and $y$ denote the coordination variable and internal variable, respectively. Constant $N_x$ and $N_y$ are dimensions of $x$ and $y$. The partition of variables is introduced in Part I of this paper. Matrix $A$, $B$, and $c$ are coefficients defining operation constraints of the subsystem. Note that variables are all bounded in power system dispatch problems due to operation limits of devices. Hence, $\Omega$ formulated in \eqref{md:ofr} is a polytope in the space of $\mathbb{R}^{N_x} \times \mathbb{R}^{N_y}$.

As per Definition 2 in Part I of this paper, the EP model of the subsystem is expressed as follows,
\begin{equation}
  \label{md:esr}
  \Phi := \left\{x \in \mathbb{R}^{N_x}: \exists y, \ \text{s.t.} \ (x,y) \in \Omega \right\}.
\end{equation}
The EP model depicts the technical and economic operation characteristics in the subspace of the coordination variable. The EP model can replace the original model of the subsystem in the coordinated optimization and can ensure the equivalent optimality in a non-iterative and privacy-protected manner, as proven in Part I.

\subsection{Representations of the EP model}
As defined in \eqref{md:esr}, EP model $\Phi$ is the geometric projection of $\Omega$ onto the space of $\mathbb{R}^{N_x}$. Since $\Omega$ is a polytope, its projection $\Phi$ is also a polytope \cite{ref:equal_set}. According to the Minkowski-Weyl Theorem \cite{ref:mw_theorem}, a polytope has the following two equivalent representations, known as the double descriptions.
\begin{itemize}
  \item Hyperplane-representation (H-rep)
\end{itemize}
\begin{equation}
  \label{md:h_rep}
  \Phi = \left\{x \in \mathbb{R}^{N_x}: \tilde{A}x \leq \tilde{c} \right\}.
\end{equation}

\begin{itemize}
  \item Vertex-representation (V-rep)
\end{itemize}
\begin{equation}
  \label{md:v_rep}
  \begin{split}
    \Phi = & \text{conv}(V)\\
    := & \left\{\sum_{i=1}^{N_v} \lambda_i \hat{x}_i: \hat{x}_i \in V, \lambda_i \ge 0, \sum_{i=1}^{N_v}\lambda_i=1 \right\}.
  \end{split}
\end{equation}
In equation \eqref{md:h_rep}, $\tilde{A}$ and $\tilde{c}$ respectively denote the coefficient and the right-hand-side of the constraints that represent hyperplanes of the polytope. In equation \eqref{md:v_rep}, function $\text{conv}(\cdot)$ denote the convex hull of a set of vectors. Set $V$ contains all the vertices of $\Phi$ and $N_v$ is the number of vertices. $\hat{x}_i$ is the $i$th vertex of $\Phi$. The H-rep and V-rep of a polytope are convertible \cite{ref:mw_theorem}.

The EP calculation is to determine either the H-rep or the V-rep of $\Phi$. Though the two representations are equivalent, they will lead to different philosophies for designing the projection algorithm. In power system optimization problems, the H-rep is mostly used to model the operation feasible region. When calculating the projection, however, we find that the V-rep is more suitable due to some features of power systems. First, the dimension of the coordination variable is lower compared with internal variables of subsystems, since networks among subsystems are relatively weak-connected compared with networks inside subsystems. Hyperplane-oriented projection methods, e.g., the FME and block elimination, calculate the projection by eliminating internal variables and thus have exponential complexity regarding the dimension of the internal variable. By identifying vertices, in contrast, the projected polytope is directly characterized in the lower-dimensional coordination space, which can lower the computational difficulty. Second, from the perspective of practical implementation, the projection algorithm is desired to terminate within a given time and yield a conservative approximation of the EP model. If the hyperplane-oriented method is terminated prematurely, some critical hyperplanes enclosing the EP model will be missed and the calculation result may contain infeasible operating points. On the contrary, if the EP is calculated by identifying vertices, the intermediate output of the algorithm will always be a subset of the exact result due to the property of the convex set. In this regard, this paper develops a vertex-oriented projection method namely the PVE to calculate the EP efficiently and practically.

\begin{algorithm}[t]
  \caption{PVE Algorithm}
  \label{alg:pve}
  \begin{algorithmic}[1]
      \STATE Initialization: $k=0, \ V = V^{(0)}$
      \REPEAT
          \STATE $k=k+1$
          \STATE Construct convex hull of existing vertices:\\ $\hat{\Phi}^{(k)} = \text{conv}(V) = \left\{x: \tilde{A}_j^{(k)} x \leq \tilde{d}_j^{(k)}, j\in[J^{(k)}]\right\}$
          \STATE Initialize set of IRs: $H^{(k)}=\emptyset$
          \FOR{$j\in[J^{(k)}]$}
              \STATE Vertex identification: solve \eqref{md:vertex} with $\alpha^\top = \tilde{A}^{(k)}_j$, obtain vertex $\hat{x}^{(k,j)}$ and IR of the vertex $\Delta h^{(k,j)}$
              \IF{$\Delta h^{(k,j)} > 0$ and $\hat{x}^{(k,j)} \notin V$}
                  \STATE Save new vertex: $V=\{V, \hat{x}^{(k,j)}\}$
                  \STATE Save IR: $H^{(k)} = \{H^{(k)}, \Delta h^{(k,j)}\}$
              \ENDIF
          \ENDFOR
          \STATE Evaluate error metric: $D^{(k)} = \max H^{(k)}$
      \UNTIL{$D^{(k)} \leq \varepsilon$}
      \STATE Output: $\hat{\Phi}=\text{conv}(V)$
  \end{algorithmic}
\end{algorithm}

\subsection{PVE Algorithm} \label{sec:pve}
The PVE algorithm calculates the projected polytope by identifying its vertices. The key of the algorithm is twofold. First, the error metric has to be carefully designed to provide a clear interpretation of the calculation results and allow the pre-specification of the error tolerance. Second, the vertices have to be identified in a proper sequence to improve computational efficiency. To this end, the Hausdorff distance is employed to measure the approximation error and a double-loop framework is developed to identify vertices of the projection along a path with the steepest descent of the approximation error. The basic idea of the algorithm is first finding vertices that are critical to the overall shape of the projection, and then expanding the convex hull of existing vertices to find new vertices outside the current approximation. The flowchart is summarized in Algorithm \ref{alg:pve}. The PVE algorithm contains two layers of iterative loops; however, this algorithm is run locally by each subsystem to calculate the EP model and its loops do not conflict with that the proposed coordinated optimization method does not require iterative information exchange among subsystems. Details of the PVE algorithm will be introduced as follows along with a numerical example illustrated in Fig. \ref{fig:pve}.

\subsubsection{Vertex identification problem}
Each vertex of the projected polytope $\Phi$ is an extreme point, which can be identified by solving the following problem,
\begin{equation}
  \label{md:vertex}
  \begin{split}
    \max_{x,y} \ & h = \alpha^\top x\\
    \text{s.t.} \ & Ax + By \leq c
  \end{split}
\end{equation}
The value vector $\alpha$ in the objective function represents the direction for identifying the vertex. 

\begin{figure}[t!]
  \centering
  \includegraphics[width=3.5in]{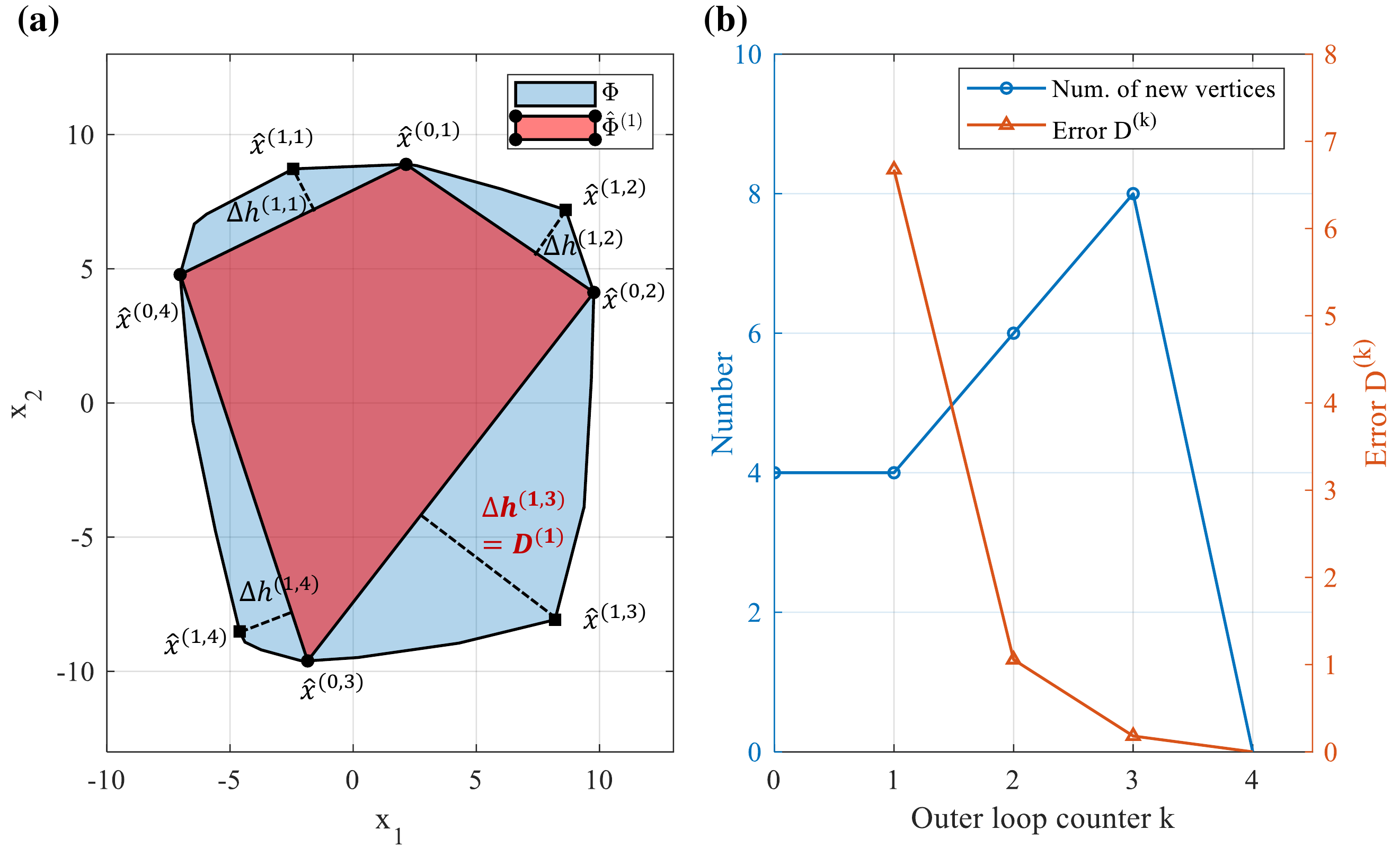}
  \caption{Illustrative example of the PVE algorithm. (a) Vertex identification in outer loop 1. (b) Number of new vertices and the error metric of each outer loop.}
  \label{fig:pve}
\end{figure}

\subsubsection{Initialization}
At least $N_x+1$ initial vertices are required to ensure the convex hull of vertices is degenerate. These vertices can be searched along each axis, i.e., solving problem \eqref{md:vertex} with $\alpha = \pm e_i$, where $e_i$ is the $i$th standard basis whose $i$th component equals 1 while others equal 0. The set of initial vertices is denoted as $V^{(0)}$. In the illustrative example in Fig. \ref{fig:pve} (a), the blue polygon is the exact projection of a randomly-generated polytope. Four initial vertices of the projection, i.e., $\hat{x}^{(0,1)},\hat{x}^{(0,2)},\hat{x}^{(0,3)}$, and $\hat{x}^{(0,4)}$, are identified for initialization.

\subsubsection{Inner loop}
The inner loop of the PVE algorithm constructs convex hull of existing vertices and identifies new vertices outside the convex hull. The convex hull of existing vertices in the $k$th outer loop is denoted as $\hat{\Phi}^{(k)}$, which is an intermediate approximation of the projection. Assume $\hat{\Phi}^{(k)}$ is enclosed by $J^{(k)}$ facets and the affine hull of the $j$th facet is $\tilde{A}_j^{(k)} x \leq \tilde{d}_j^{(k)}$. New vertices are searched along outer normal directions of facets of $\hat{\Phi}^{(k)}$. For the $j$th facet, its outer normal vector is $\tilde{A}_j^{(k)}$. Thereby, the $j$th inner loop identifies the new vertex by solving problem \eqref{md:vertex} with $\alpha^\top = \tilde{A}^{(k)}_j$. The newly identified vertex is the optimal solution of problem \eqref{md:vertex}, denoted as $\hat{x}^{(k,j)}$. The improvement ratio (IR) is calculated as \eqref{def:ir} to measure the contribution of vertex $\hat{x}^{(k,j)}$ to improve the current approximation.
\begin{equation}
  \label{def:ir}
  \Delta h^{(k,j)} = \frac{(\tilde{A}^{(k)}_j)^\top \hat{x}^{(k,j)} - \tilde{d}^{(k,j)}}{\| \tilde{A}^{(k)}_j\|}.
\end{equation}
The IR measures the distance from vertex $\hat{x}^{(k,j)}$ to the $j$th facet of $\hat{\Phi}^{(k)}$ and is non-negative. If $\Delta h^{(k,j)}>0$, indicating that $\hat{x}^{(k,j)}$ is outside $\hat{\Phi}^{(k)}$ and will contribute to improve the current approximation, then $\hat{x}^{(k,j)}$ will be appended to the vertex set $V$ and $\Delta h^{(k,j)}$ will be recorded in set $H^{(k)}$. If $\Delta h^{(k,j)} = 0$, indicating that $\hat{x}^{(k,j)}$ is on the facet of $\hat{\Phi}^{(k)}$ and will not contribute to improve the current approximation, then this vertex will be omitted.

The example in Fig. \ref{fig:pve} (a) exhibits the vertex identification in the first outer loop ($k=1$). The convex hull of existing vertices is $\hat{\Phi}^{(1)}$, as the red polygon shows. Vertices $\hat{x}^{(1,1)},\hat{x}^{(1,2)},\hat{x}^{(1,3)}$, and $\hat{x}^{(1,4)}$ are identified along the outer normal vector of the four facets of $\hat{\Phi}^{(1)}$, respectively. The Distance of each vertex to the corresponding facet of $\hat{\Phi}^{(1)}$ is the IR, as the dotted line mark in Fig. \ref{fig:pve} (a).

\subsubsection{Outer loop and termination criterion}
The outer loop evaluates the error of the current approximation and compares it with the pre-specified error tolerance to decide whether to terminate the algorithm. The Hausdorff distance between the real projection $\Phi$ and the current polytope $\hat{\Phi}^{(k)}$ is employed as the error metric, which is defined as follows \cite{ref:hd},
\begin{equation}
  D^{(k)} := \max_{x_1 \in \Phi} \min_{x_2 \in \hat{\Phi}^{(k)}} \|x_2-x_1\|_2.
\end{equation}
The interpretation of $D^{(k)}$ is the maximum distance from points in $\Phi$ to polytope $\hat{\Phi}^{(k)}$. Note that $\hat{\Phi}^{(k)}$ is the convex hull of existing vertices of $\Phi$, therefore $\hat{\Phi}^{(k)} \subset \Phi$. Hence, $D^{(k)}$ will be the maximum distance from vertices of $\Phi$ to polytope $\hat{\Phi}^{(k)}$. According to the definition in \eqref{def:ir}, the distance from each identified vertex of $\Phi$ to polytope $\hat{\Phi}^{(k)}$ is actually the corresponding IR. Thereby, the error metric can be evaluated as follows,
\begin{equation}
  \label{eq:cal_hd}
  \begin{split}
    D^{(k)} & = \max_{x_1 \in V} \min_{x_2 \in \hat{\Phi}^{(k)}} \|x_2-x_1\|_2 \\
    & =\max_{j} \Delta h^{(k,j)}\\
    & = \max H^{(k)}.
  \end{split}
\end{equation}
From the above derivation, the Hausdorff distance can be evaluated by selecting the largest value from a set of scalars, which will consume much less computation effort compared to error metrics based on the region volume. If $D^{(k)}$ is no larger than the pre-specified error tolerance $\varepsilon$, the algorithm terminates and outputs the convex hull of existing vertices as the approximation of the real projection. The error metric based on the Hausdorff distance ensures that the maximum deviation of the approximation is within $\varepsilon$. If the termination criterion is not met, the algorithm moves into the next outer loop to identify more vertices outside the current approximation. 

In Fig. \ref{fig:pve} (a), the error metric $D^{(1)}$ for outer loop $k=1$ equals the IR of vertex $\hat{x}^{(1,3)}$, which has the farthest distance to $\hat{\Phi}^{(1)}$. Fig. \ref{fig:pve} (b) exhibits the number of new vertices and the error metric in each outer loop of the PVE algorithm. The algorithm terminates at the $4$th outer loop when $D^{(4)}=0$. A total of 22 vertices of the projection are identified. As shown in the figure, the error metric keeps decreasing along the steepest path during the algorithm, indicating that vertices that have significant influences on the projection are identified with higher priority.

With the Hausdorff distance-based error metric, the process of the PVE algorithm can be interpreted as searching vertices that make the approximation error decline fast. This error metric also has another two advantages. First, the Hausdorff distance can be evaluated conveniently according to \eqref{eq:cal_hd}, which does not bring additional computational complexity. Second, the evaluation of the Hausdorff distance does not rely on the ground truth of the projection result and thus, the proposed error metric can be used to control the accuracy during the projection algorithm. In contrast, other error metrics, e.g., the volume difference metric and the Jaccard similarity, rely on the knowledge of the exact projection result and consequently, are hard to calculate and cannot be used to control the accuracy of the projection algorithm.

\subsection{Discussions}
\subsubsection{Properties of the PVE algorithm}
The convergence of the PVE algorithm and the conservatism of the approximation are guaranteed by the following theorem.
\begin{theorem}
  If error tolerance $\varepsilon>0$, the output of the PVE algorithm is a conservative approximation of the exact projection. If $\varepsilon=0$, the exact projection can be obtained via the PVE algorithm within finite calculations.
\end{theorem}

\begin{proof}
  In each round of the outer loop of the PVE algorithm, the convex hull of existing vertices is an intermediate approximation of the projection. Since the projected region is a polytope, the convex hull of its vertices is a subset of the exact projection, i.e., $\hat{\Phi}^{(k)} \subset \Phi$ for any $k$. Hence, if $\varepsilon>0$, a conservative approximation of the exact projection will be obtained with the maximum Hausdorff distance no larger than $\varepsilon$. If $\varepsilon$ is set 0, indicating that no point of $\Phi$ is outside $\hat{\Phi}^{(k)}$ when the algorithm terminates, then $\Phi \subset \hat{\Phi}^{(k)}$. Note that $\hat{\Phi}^{(k)} \subset \Phi$ and thus, $\hat{\Phi}^{(k)} = \Phi$. Since the number of vertices of a polytope is finite, the algorithm will terminate within finite steps. 
\end{proof}

\subsubsection{Accelerating strategies}
Three strategies are proposed to further accelerate the PVE algorithm. The first is to accelerate the solving of the vertex identification problem in \eqref{md:vertex}. Note that only the value vector in the objective function is varying for identifying different vertices, the structure and parameters of constraints are invariant. Thereby, the redundant constraint elimination technique and the warm simplex basis technique can be used to accelerate the solution of \eqref{md:vertex}. The redundant constraint elimination technique removes constraints that do not impact the feasible region to reduce the problem scale. Typical methods can be found in literature \cite{ref:redundant1} and \cite{ref:redundant2}. The warm simplex basis can be specified for problem \eqref{md:vertex} using the solution of the adjacent vertex that is on the same facet of which the outer normal vector is used as the searching direction. The specification of the start basis can reduce the iteration number of the simplex algorithm and thus accelerating the vertex identification. The second strategy is to accelerate the inner loop of the PVE algorithm by parallel computing. In line 7 of the algorithm, the identification of each vertex is independent, which can be calculated in parallel to save the computation time. The third is to dynamically update the convex hull when adding new vertices in each outer loop (line 4 of Algorithm \ref{alg:pve}), instead of building a new convex hull for all vertices. The updating of the convex hull can be realized by the quick hull method \cite{ref:qhull}, which is not only computationally efficient but also capable of constructing high-dimensional convex hulls.

\subsubsection{Analysis of adaptability}
First, for the problem where the projection region is a general convex set, the proposed PVE algorithm can also be applied. Perimeter points of the projection region will be identified by the algorithm, and the convex hull of these perimeter points is an inner approximation of the projection result. The approximation error measured by the Hausdorff distance is less than $\varepsilon$. Second, if the projection region is non-convex, the PVE algorithm cannot be directly applied. A potential solution is decomposing the non-convex projection region into the union of several convex sub-regions and using the PVE algorithm to identify each sub-region. Third, for the problem involving multiple time intervals, the coordination variable may have a high dimension and the direct application of the PVE algorithm will be inefficient. There are strategies to decompose the time-coupled projection problem into a series of lower-dimensional subproblems regarding each single time interval and adjacent intervals \cite{ref:lin}. After decomposition, each low-dimensional projection can be calculated efficiently by the PVE algorithm.

\section{Application to Multi-Area Coordinated Optimal Dispatch}\label{sec:use1}
\subsection{Problem Formulation}
Power systems in different areas are interconnected physically through inter-area transmission lines, but different areas may be operated by different system operators. The multi-area coordinated optimal dispatch (MACOD) is thus becoming indispensable for the economic and secure operation of large-scale power systems. The MACOD minimizes the total operation cost of the multi-area system in a decomposed manner. 

Use $P$, $\pi$, and $\theta$ to represent variables of active power, operation cost, and voltage phase angle, respectively. Use superscript $G$, $D$, $B$, $RN$, $MA$, and $TL$ to label generation, load, power exchange at boundary node, regional network, multi-area system, and tie-line, respectively. Let $r\in \mathcal{R}^{RN}$ and $s \in \mathcal{S}$ index areas and tie-lines, respectively. Let $n \in \mathcal{N}^{RN}_r$ and $l\in \mathcal{L}^{RN}_r$ index network nodes and internal branches of the $r$th area, respectively. Then the objective function of a basic MACOD is as follows,
\begin{equation}
  \label{macd:obj}
  \min \quad \pi^{MA} = \sum_{r \in \mathcal{R}^{RN}} \pi^{RN}_r
\end{equation}
where $\pi^{RN}$ is the operation cost of area $r$ and $\pi^{MA}$ is the total cost of the multi-area system.

The following constraints need to be satisfied for area $r$.
\begin{subequations}
  \label{macd:cst_area}
  \begin{align}
    & \overline{\pi}^{RN}_r \ge \pi^{RN}_r \ge \sum_{n \in \mathcal{N}^{RN}_r} \pi^{RN,G}_{r,n},\label{macd:cost_1}\\
    & \pi^{RN,G}_{r,n} \ge a^{RN,G}_{r,n,i} P^{RN,G}_{r,n} + b^{RN,G}_{r,n,i}, \forall n,i,\label{macd:cost_2}\\
    & \sum_{n \in \mathcal{N}^{RN}_r} P^{RN,G}_{r,n} = \sum_{n \in \mathcal{N}^{RN}_r} P^{RN,D}_{r,n} + \sum_{n \in \mathcal{N}^{RN,B}_r} P^{RN,B}_{r,n}, \label{macd:bal}\\
    & \begin{aligned}
        -\overline{F}^{RN}_{r,l} \leq & \sum_{n \in \mathcal{N}^{RN}_r} T_{r,l,n}P^{RN,G}_{r,n} - \sum_{n \in \mathcal{N}^{RN}_r} T_{r,l,n}P^{RN,D}_{r,n} \\
        & - \sum_{n \in \mathcal{N}^{RN,B}_r} T_{r,l,n}P^{RN,B}_{r,n} \leq \overline{F}^{RN}_{r,l}, \forall l,
      \end{aligned} \label{macd:branch}\\
    & \underline{P}^{RN,G}_{r,n} \leq P^{RN,G}_{r,n} \leq \overline{P}^{RN,G}_{r,n}, \forall n.\label{macd:gen}
  \end{align}
\end{subequations}
In the above model, equation \eqref{macd:cost_1} models the operation cost of area $r$ in the epigraph form, where $\pi^{RN,G}_{r,n}$ is the cost of generator $n$ and $\overline{\pi}^{RN}_r$ is a large enough constant to bound $\pi^{RN}_r$. Equation \eqref{macd:cost_2} is the epigraph of the piece-wise linear cost function of generator $n$, where $a^{RN,G}_{r,n,i}$ and $b^{RN,G}_{r,n,i}$ are coefficients of each bidding segment, respectively. Equation \eqref{macd:bal}, \eqref{macd:branch}, and \eqref{macd:gen} enforce the power balance, transmission limits, and generation limits of area $r$. In the equations, $T_{r,l,n}$ is the power transfer distribution factor of node $n$ to branch $l$, $\overline{F}^{RN}_{r,l}$ is the capacity of branch $l$, $\underline{P}^{RN,G}_{r,n}$ and $\overline{P}^{RN,G}_{r,n}$ are lower and upper bounds of generation output, respectively. 

The operating constraints of tie-lines are as follows,
\begin{subequations}
  \label{macd:cst_tie}
  \begin{align}
    & P^{TL}_{s} = \frac{1}{x_s}\left(\theta_{r^F_s} - \theta_{r^T_s} \right), \forall s,\label{macd:tie1}\\
    & P^{RN,B}_{r,n} = \sum_{\substack{r^F_s = r, \\ n^F_s = n}} P^{TL}_s - \sum_{\substack{r^T_s = r,\\ n^T_s = n}} P^{TL}_s, \forall n, r,\label{macd:tie2}\\
    & \underline{F}^{TL}_s \leq P^{TL}_{s} \leq \overline{F}^{TL}_s, \forall s. \label{macd:tie3}
  \end{align}
\end{subequations}
Equation \eqref{macd:tie1} is the direct-current power flow model for tie-lines, where $x_s$ is the reactance of tie-line $s$, $r^F_s$ and $r^T_s$ denote the from and end region of the tie-line, respectively. In this work, each area is aggregated as a node when evaluating power flows on tie-lines, which is widely accepted in the literature \cite{ref:node1} and is also implemented in the flow-based market integration in Europe \cite{ref:node2}. Equation \eqref{macd:tie2} maps tie-line power flows to boundary power injections of each area, where $n^F_s$ and $n^T_s$ denote the from and end node of tie-line $s$, respectively. Equation \eqref{macd:tie3} enforces transmission limits of tie-lines.

\subsection{EP-Based Solution for MACOD}\label{cp:esr_macd}
In the MACOD problem, variable $P^{RN,B}_{r,n}$ appears in constraints of both the regional system and tie-line system, which restricts the independent optimization of each area. Based on the EP theory proposed in Part I of this paper, the MACOD problem can be solved in a decomposed manner without iterative information exchange. 

From the perspective of primal decomposition, $P^{RN,B}_{r}=(P^{RN,B}_{r,1},\cdots,P^{RN,B}_{r,|\mathcal{N}^{RN,B}_r|})$ and $\pi^{RN}_r$ are seen as the coordination variable corresponding to area $r$, while $P^{RN,G}_{r}=(P^{RN,G}_{r,1},\cdots,P^{RN,G}_{r,|\mathcal{N}^{RN,G}_r|})$ is the internal variable. According to \eqref{md:ofr}, the operation feasible region of each area is as follows,
\begin{equation}
  \Omega^{RN}_r = \left\{(P^{RN,B}_r, \pi^{RN}_r, P^{RN,G}_{r}): \text{Eq.} \ \eqref{macd:cst_area}\right\}.
\end{equation}
According to \eqref{md:esr}, the EP model of the area is the projection of $\Omega^{RN}_r$ onto the subspace of coordination variable, i.e.,
\begin{equation}
  \label{md_area:esr}
  \begin{split}
    \Phi^{RN}_r = \{& (P^{RN,B}_r, \pi^{RN}_r): \exists P^{RN,G}_{r}, \\
    & \text{s.t. } (P^{RN,B}_r, \pi^{RN}_r, P^{RN,G}_{r}) \in \Omega^{RN}_r \}.
  \end{split}
\end{equation}
The EP model of the regional system contains all values of $(P^{RN,B}_r, \pi^{RN}_r)$ that can be met by at least one generation schedule subject to regional operation constraints with generation cost no larger than $\pi^{RN}_r$. Note that the objective of the MACOD is to minimize the sum of $\pi^{RN}_r$. Using the EP model as a substitute for the regional system model \eqref{macd:cst_area} in the MACOD, the cost variable $\pi^{RN}_r$ will be minimized and the decision result of $P^{RN,B}_r$ will be ensured to be feasible for the regional system, which leads to the coordinated optimality of multiple areas. Theoretical proof of the optimality is in Theorem 3 in Part I of this paper.

Then the EP-based MACOD is realized following three steps. 
\begin{enumerate}[1)]
  \item Equivalent projection. Each area calculates the EP model $\Phi^{RN}_r$ of the local system according to the definition in \eqref{md_area:esr}, which can be realized by the PVE algorithm introduced in Section \ref{sec:pve}. Whereafter, $\Phi^{RN}_r$ is submitted to the multi-area coordinator.
  \item Coordinated optimization. The coordinator solves the EP-based MACOD problem formulated as $\min \{\sum_{r} \pi^{RN}_r : \text{Eq. } \eqref{macd:cst_tie}, (P^{RN,B}_r, \pi^{RN}_r) \in \Phi^{RN}_r , \forall r \in \mathcal{R}^{RN} \}$. The optimal value $(\hat{P}^{RN,B}_r, \hat{\pi}^{RN}_r)$ is then published to each area.
  \item Regional system operation. Each area fixes $\hat{P}^{RN,B}_r$ as the boundary condition and solves the local optimal dispatch problem formulated as $\min \{\pi^{RN}_r : \text{Eq.} \eqref{macd:cst_area}, P^{RN,B}_r = \hat{P}^{RN,B}_r\}$.
\end{enumerate}
The above process is executed sequentially, which naturally overcomes drawbacks caused by repetitive iterations of conventional coordinated optimization methods.

\subsection{Case Study}
The EP of the regional transmission system is visualized, and the computational performance of the EP calculation and the EP-based MACOD is tested. Case studies in this paper are simulated on a Lenovo ThinkPad X13 laptop with an Intel Core i7-10510U 1.80-GHz CPU. Algorithms are programmed in MATLAB R2020a with GUROBI V9.0.0.
\subsubsection{EP of regional system}
\begin{figure}[t!]
  \centering
  \includegraphics[width=3.5in]{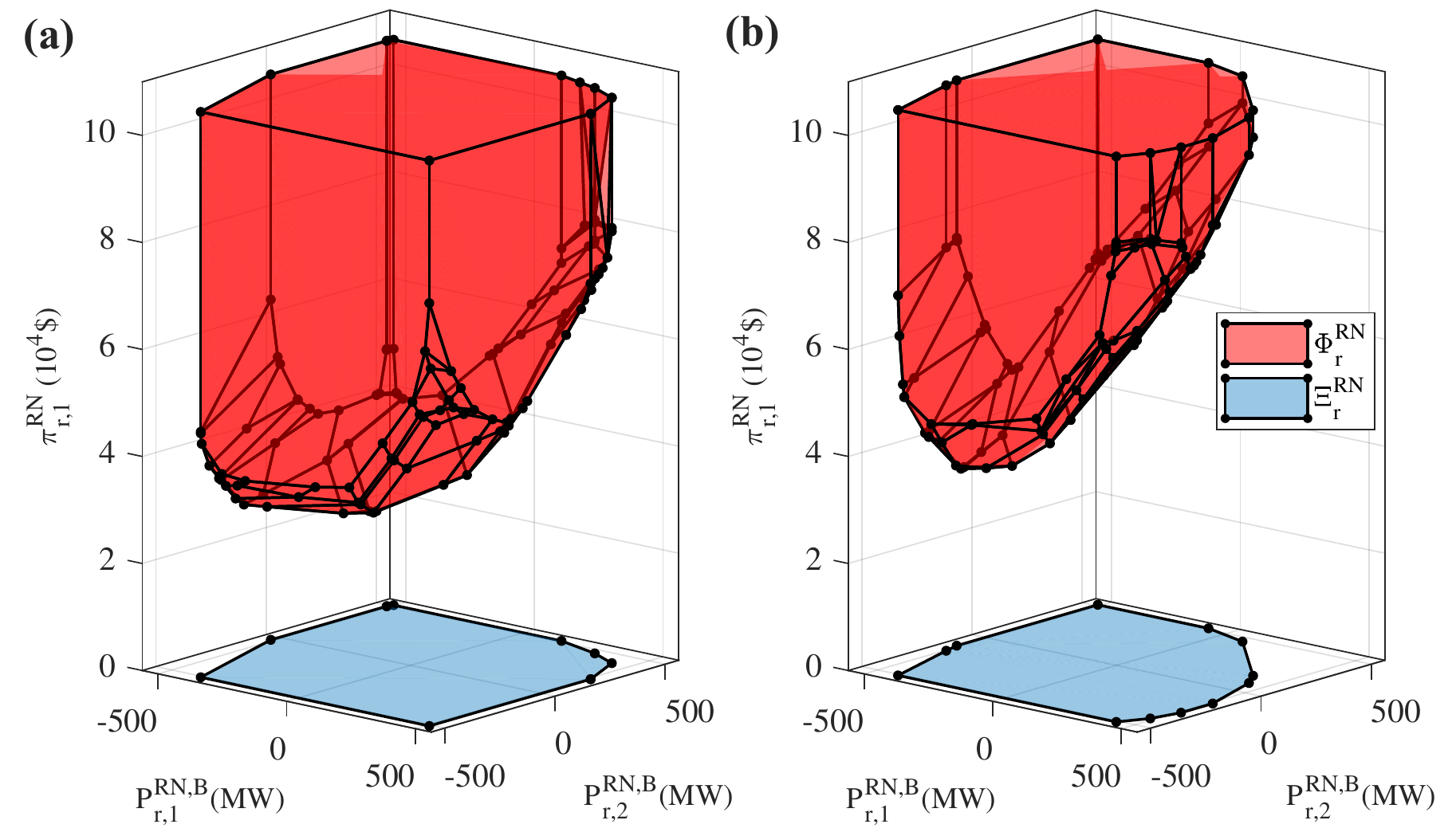}
  \caption{EP models of the IEEE-24 system at (a) valley hour and (b) peak hour.}
  \label{fig:reg_esr}
\end{figure}
We employ the IEEE-24 system for visualizing the EP model of the regional transmission system. Two tie-lines are connected to node 1 and node 3 of the system, and each of them has the capacity of $\SI{510}{MW}$ (15\% of the total generation capacity). The results are shown in Fig \ref{fig:reg_esr}, where subplots (a) and (b) exhibit EP models at the valley-load hour (77\% of the peak load) and peak-load hour, respectively. Red regions are EP models of the system, which are 3-dimensional polytopes. As can be seen, the volume of the EP model at the peak hour is smaller than that at the valley hour. This is because the increase of load occupies the export capacity and upraising the power supply cost. The projection of the EP model on the subspace of $(P^{RN,B}_{r,1}, P^{RN,B}_{r,2})$, denoted as $\Xi^{RN}_r$, is also plotted, which characterizes the admissible set of tie-lie flows that can be executed by the regional system.

\subsubsection{EP calculation}
\begin{table}[t!]
  \centering
  \begin{threeparttable}[c]
  \caption{Computational Efficiency of EP Calculation}
  \label{tab:esr_eff}
  \renewcommand\arraystretch{1.3}
  \setlength{\tabcolsep}{2.5mm}{
  \begin{tabular}{ccccc}
      \toprule
      System & \makecell{$|\mathcal{N}^B_r|$} & \makecell{Time of \\ FME (s)} & \makecell{Time of \\PVE (s)} & \makecell{Rate of \\ model reduction}\\
      \midrule
      \multirow{2}{*}{IEEE-24}& 3 & \textgreater1200  & 0.34  &  95.4\% \\
                              & 6 & \textgreater1200  & 3.54  &  92.7\% \\  
      \multirow{2}{*}{SG-200} & 3 & \textgreater1200  & 0.46  &  99.8\% \\
                              & 6 & \textgreater1200  & 2.69  &  99.6\% \\  
      \multirow{2}{*}{SG-500} & 3 & \textgreater1200  & 0.47  &  99.9\% \\
                              & 6 & \textgreater1200  & 2.99  &  99.6\% \\  
      \bottomrule
  \end{tabular}
  }
  \end{threeparttable}
\end{table}

We test the computational performance of the proposed PVE algorithm for EP calculation based on the IEEE-24 test system, the 200 node-synthetic grid (SG-200), and the 500 node-synthetic grid (SG-500) \cite{ref:synthetic}. Cases with 3 and 6 tie-lines are considered for each test system. The FME is employed as the benchmark method. As summarized in TABLE \ref{tab:esr_eff}, the FME fails to obtain the EP model within 1200s even for the small-scale IEEE-24 system. The proposed PVE algorithm, in contrast, can yield the EP model within 4s for all the 6 cases, which is acceptable for practical application. Measure the scale of the regional system model \eqref{macd:cst_area} by the production of numbers of variables and constraints. With the EP, the model scale of the regional system is reduced by 92.7\%$\sim$99.9\%, which will greatly alleviate the communication and computation burden of the coordinated dispatch of multiple areas as well as protecting the private information of regional systems. 

\subsubsection{EP-based MACOD}
\begin{table}[t]
  \centering
  \caption{Computational Efficiency of EP-based MACOD}
  \label{tab:macd_eff}
  \renewcommand\arraystretch{1.3}
  \setlength{\tabcolsep}{2mm}{
  \begin{tabular}{ccccc}
      \toprule
      \multirow{2}{*}{System} & \multirow{2}{*}{\makecell{Time of joint\\ optimization (s)}} & \multicolumn{3}{c}{Time of EP-based coordination (s)} \\
      \cmidrule{3-5}
       & &  \makecell{System \\ reduction} & \makecell{Coordinated \\ optimization} & Total \\
      \midrule
      $20\times \text{SG-200}$ & 0.42 & 0.38 & 0.11 & 0.49\\
      $40\times \text{SG-200}$ & 0.61 & 0.32 & 0.10 & 0.42\\
      $20\times \text{SG-500}$ & 0.51 & 0.49 & 0.09 & 0.58\\
      $40\times \text{SG-500}$ & 0.78 & 0.51 & 0.10 & 0.61\\
      \bottomrule
  \end{tabular}
  }
\end{table}

Four test systems composed of 20 and 40 synthetic grids of different scales are constructed according to the tie-line topology in \cite{ref:tan_macd}. The MACOD is solved by the joint optimization and the EP-based decomposed solution, respectively. The optimization results from the two methods are verified to be identical, which validates the accuracy of the proposed coordination method. According to the process introduced in Section \ref{cp:esr_macd}, the time for system reduction depends on the region that takes the longest computation time, since EP models of different regions are calculated in parallel. The total time to obtain the optimal tie-line scheduling is the sum of the system reduction time and the coordinated optimization time. Results are summarized in TABLE \ref{tab:macd_eff}. As is shown, the total computation time of the EP-based MACOD is comparable with that of the joint optimization. With the EP, the time consumed by the coordinated optimization is reduced by more than 73.8\% compared with the joint optimization. This is because the scale of the coordination problem is significantly reduced with the EP of each region. In addition to the computation efficiency, the primary advantages of the proposed coordination method are avoiding private data disclosure compared with the joint optimization and avoiding iterative information exchange compared with conventional coordinated optimization methods.


\section{Application to Transmission-Distribution Coordinated Optimal Dispatch}\label{sec:use2}
\subsection{Problem Formulation}
Distribution networks installed with DERs can provide flexibility to the transmission system by adjusting the power exchange at the substation, which calls for the transmission-distribution coordinated optimal dispatch (TDCOD). In this study, optimal active power dispatch is considered for the transmission system, while the co-optimization of active and reactive power is considered for the distribution system with voltage limits. Let $P$, $Q$, $V$, and $\pi$ denote variables of active power, reactive power, voltage magnitude, and operation cost, respectively. Similar to the notation in Section \ref{sec:use1}, use superscript $G$, $D$, $F$ to label generation, load, and power flow, respectively. Use superscript $TN$ and $DN$ to label variables of transmission and distribution systems, respectively. Let $r\in\mathcal{R}^{DN}$ index distribution networks participate in the coordinated dispatch. Let $n\in\mathcal{N}^{DN}_r$ and $l\in\mathcal{L}^{DN}_r$ index nodes and branches of distribution network $r$. Let $n\in\mathcal{N}^{TN}$ and $l\in\mathcal{L}^{TN}$ index nodes and branches of the transmission network.

The objective of the TDCOD is minimizing the total operation cost of the transmission and distribution systems,
\begin{equation}
  \label{tdcd:obj}
  \min \ \pi^{TD} = \sum_{n\in\mathcal{N}^{TN}}C_n(P^{TN,G}_n) + \sum_{r\in\mathcal{R}^{DN}} \pi^{DN}_r
\end{equation}

The power balance constraint, transmission limits, and generation limits are considered for the transmission network,
\begin{subequations}
  \label{tdcd:tn}
  \begin{align}
    & \sum_{n\in\mathcal{N}^{TN}}P^{TN,G}_n + \sum_{r\in\mathcal{R}^{DN}}P^{DN}_{r,0} = \sum_{n\in\mathcal{N}^{TN}}P^{TN,D}_n, \label{tdcd:tn1}\\
    & \begin{aligned}
      -\overline{F}^{TN}_l \leq & \sum_{n\in\mathcal{N}^{TN}} T^{TN}_{l,n}\left(P^{TN,G}_n-P^{TN,D}_n \right) \\
      & +\sum_{r\in\mathcal{R}^{DN}} T^{TN}_{l,r} P^{DN}_{r,0} \leq \overline{F}^{TN}_l, \forall l,
    \end{aligned} \label{tdcd:tn2}\\
    & \underline{P}^{TN,G}_n \leq P^{TN,G}_n \leq \overline{P}^{TN,G}_n , \forall n. \label{tdcd:tn3}
  \end{align}
\end{subequations}

The operation constraints of the $r$th distribution network is as follows,
\begin{subequations}
  \label{tdcd:dn}
  \begin{align}
    & \overline{\pi}^{DN}_r \ge \pi^{DN}_r \ge \sum_{n\in\mathcal{N}^{DN}_r} \pi_{r,n}^{DN,G},\label{tdcd:dn_cost1}\\
    & \pi_{r,n}^{DN,G} \ge a^{DN,G}_{r,n,i} P^{DN,G}_{r,n} + b^{DN,G}_{r,n,i}, \label{tdcd:dn_cost2}\\
    & \begin{aligned}
      & P^{DN,G}_{r,n} + \sum_{m \in \mathcal{A}_n} P^{DN,F}_{r,mn}= P^{DN,D}_{r,n}, P^{DN}_{r,0} =  \sum_{m \in \mathcal{A}_0} P^{DN,F}_{r,m0}, \\
      & Q^{DN,G}_{r,n} + \sum_{m \in \mathcal{A}_n} Q^{DN,F}_{r,mn}= Q^{DN,D}_{r,n}, Q^{DN}_{r,0} =  \sum_{m \in \mathcal{A}_0} Q^{DN,F}_{r,m0}, 
    \end{aligned} \label{tdcd:dn_bal}\\
    & V^2_{r,m} - V^2_{r,n} = 2(r_{r,mn} P^{DN,F}_{r,mn} + x_{r,mn} Q^{DN,F}_{r,mn}), \label{tdcd:dn_pf}\\
    & (\cos \frac{2k\pi}{N_s})P^{DN,F}_{r,mn}+(\sin \frac{2k\pi}{N_s})Q^{DN,F}_{r,mn} \leq (\cos \frac{\pi}{N_s})\overline{F}^{DN}_{r,mn},\label{tdcd:dn_cap}\\
    & \underline{V}^2_{r,n} \leq V^2_{r,n} \leq \overline{V}^2_{r,n}, \label{tdcd:dn_vol}\\
    & \underline{P}^{DN,G}_{r,n} \leq P^{DN,G}_{r,n} \leq \overline{P}^{DN,G}_{r,n},\underline{Q}^{DN,G}_{r,n} \leq Q^{DN,G}_{r,n} \leq \overline{Q}^{DN,G}_{r,n}. \label{tdcd:dn_gen}
  \end{align}
\end{subequations}
Equation \eqref{tdcd:dn_cost1} and \eqref{tdcd:dn_cost2} model operation costs of the distribution system and DER $n$ in the epigraph form, where $\overline{\pi}^{DN}_r$ is a large enough constant, $a^{DN,G}_{r,n,i}$ and $b^{DN,G}_{r,n,i}$ are coefficients of the $i$th segment of the piece-wise linear cost function of DER $n$. Equation \eqref{tdcd:dn_bal} enforces nodal power balance, where $P^{DN,F}_{r,mn}$ and $Q^{DN,F}_{r,mn}$ respectively denote active and reactive power flow on the feeder from node $m$ to node $n$, $\mathcal{A}_n$ is the set of nodes adjacent to node $n$. $P^{DN}_{r,0}$ and $Q^{DN}_{r,0}$ are net active and reactive power the distribution network injects into the transmission system. Equation \eqref{tdcd:dn_pf} is the simplified DistFlow model of the distribution network \cite{ref:distflow}, which incorporates the reactive power and can be used to the distribution network with large R/X ratio. Equation \eqref{tdcd:dn_cap} enforces the transmission limit of branch $mn$ with piece-wise inner approximation and $\overline{F}^{DN}_{r,mn}$ is the capacity of the branch. Equation \eqref{tdcd:dn_vol} and \eqref{tdcd:dn_gen} enforce limits on nodal voltage magnitudes and power output of DERs, respectively. Taking $V^2_{r,n}$ as the independent variable, the above constraints are linear.

\subsection{EP-based Solution for TDCOD}
In problem \eqref{tdcd:obj}-\eqref{tdcd:dn}, the optimization of transmission and distribution systems are coupled by $P^{DN}_{r,0}$ and $\pi^{DN}_r$. For the $r$th distribution network, take $x^{DN}_r = (P^{DN}_{r,0}, \pi^{DN}_r)$ as the coordination variable and take $y^{DN}_r = (Q^{DN}_{r,0}, P^{DN,G}_{r}, Q^{DN,G}_{r}, P^{DN,F}_{r}, Q^{DN,F}_{r}, V^2_r)$ as the internal variable. The operation feasible region of the distribution network, denoted as $\Omega^{DN}_r$, is the set of $(x^{DN}_r,y^{DN}_r)$ subject to constraints in \eqref{tdcd:dn}. Then the EP model of the distribution network can be formulated according to the definition in \eqref{md:esr}, denoted as $\Phi^{DN}_r$. The EP model is the projection of $\Omega^{DN}_r$ onto the subspace of $x^{DN}_r$. Note that constraints \eqref{tdcd:dn} are linear and thus, both $\Omega^{DN}_r$ and its projection $\Phi^{DN}_r$ are polytopes, and $\Phi^{DN}_r$ can be calculated by the proposed PVE algorithm. The EP model $\Phi^{DN}_r$ contains all possible combinations of $P^{DN}_{r,0}$ and $\pi^{DN}_r$ that can be executed by the distribution network with at least one internal generation schedule satisfying constraints in \eqref{tdcd:dn}. Using EP models of distribution networks to replace their original model in the TDCOD problem, the operation cost $\pi^{DN}_r$ of the distribution network will be minimized as in \eqref{tdcd:obj}, and the active power exchange $P^{DN}_{r,0}$ determined by the transmission system operator will be ensured to be feasible for the distribution network. Hence, the EP-based coordinated optimization of transmission and distribution systems will be identical to the joint optimization.

The process of the EP-based TDCOD contains three successive steps.
\begin{enumerate}[1)]
  \item Equivalent projection. Each distribution system calculates the EP model $\Phi^{DN}_r$ and reports it to the transmission system operator.
  \item Coordinated optimization. The transmission system operator solves the EP-based TDCOD problem, i.e., minimizing the objective function in \eqref{tdcd:obj} subject to constraint \eqref{tdcd:tn} and $(P^{DN}_{r,0}, \pi^{DN}_r) \in \mathcal{R}^{DN}$. The optimal value $(\hat{P}^{DN}_{r,0}, \hat{\pi}^{DN}_r)$ is then published to each distribution system.
  \item Distribution system operation. Each DSO fixes $\hat{P}^{DN}_{r,0}$ as the boundary condition and optimally dispatches the local system by solving the problem $\min \{\pi^{DN}_r : \text{Eq.} \eqref{tdcd:dn}, P^{DN}_{r,0}=\hat{P}^{DN}_{r,0} \}$.
\end{enumerate}
The above coordination process does not require iterations between the transmission and distribution systems. This coordination scheme is also compatible with the existing transmission system dispatch that schedules the optimal output of resources with their submitted information.

\subsection{Case Study}
\subsubsection{EP of the distribution network}
\begin{figure}[t!]
  \centering
  \includegraphics[width=3.5in]{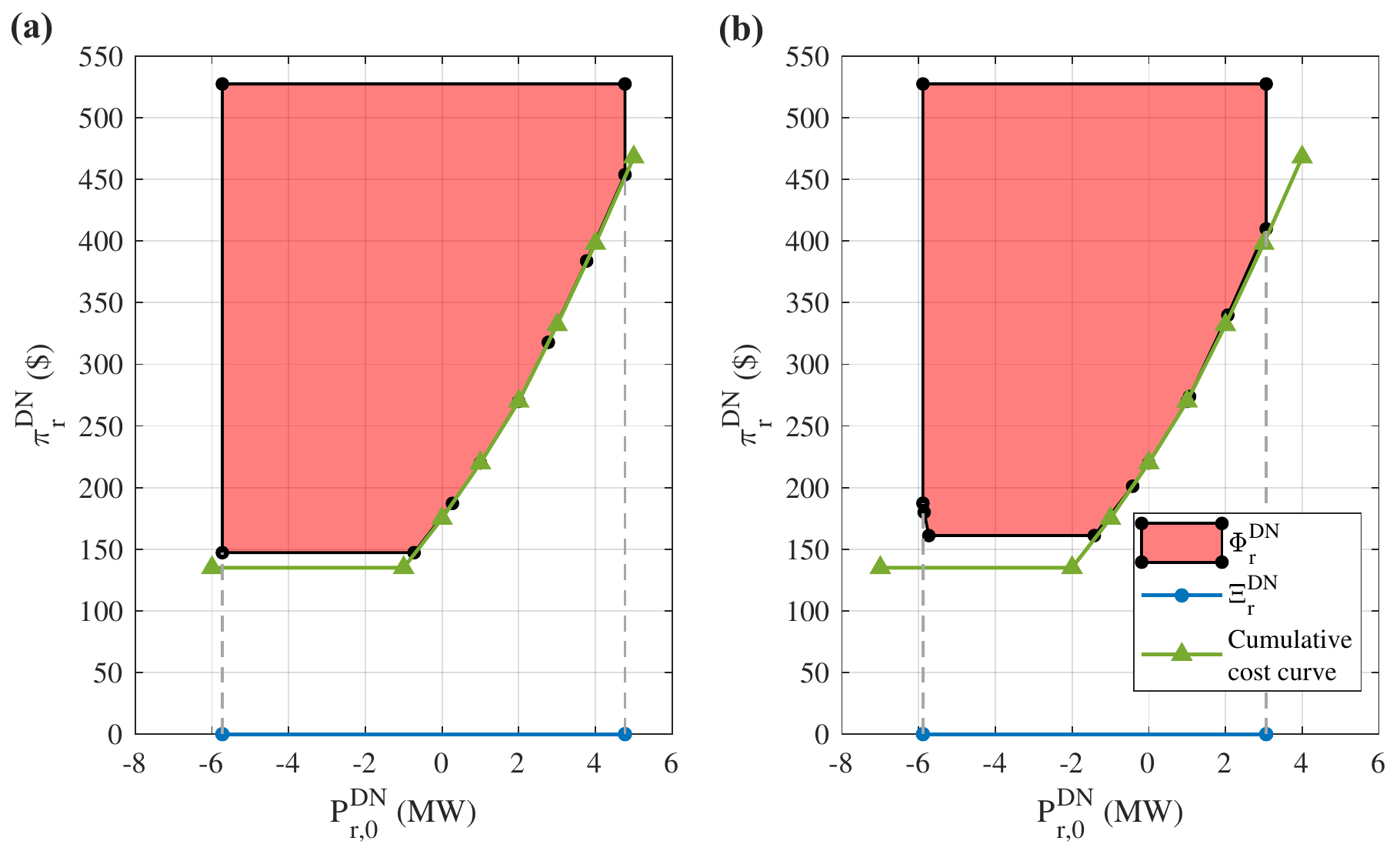}
  \caption{EP models of the IEEE-13 distribution network at (a) valley load and (b) peak load.}
  \label{fig:dist_esr}
\end{figure}
The modified IEEE-13 test feeder with 6 DERs in reference \cite{ref:tan_rcc} is used to demonstrate the EP of the distribution network. Red polygons in Fig. \ref{fig:dist_esr} exhibit EP models of the test system at the valley and peak hours, which are 2-dimensional regions regarding the net power exchange $P^{DN}_{r,0}$ and the operation cost $\pi^{DN}_r$. The projection of the EP model on the subspace of $P^{DN}_{r,0}$ is denoted as $\Xi^{DN}_r$, which characterizes the range of flexibility that the distribution network can provide. The green curves with triangular markers are cumulative cost functions of DERs. As can be seen, the cumulative cost curve is beneath the EP model and has a wider range of net power exchange. This is because network limits are omitted when generating the cumulative cost curve. Hence, only aggregating cost curves of DERs for transmission-level dispatch is not enough, the EP of the distribution system incorporating both cost functions and network constraints is needed for the optimal coordination between transmission and distribution systems.

\subsubsection{EP calculation}
\begin{table}[t!]
  \centering
  \caption{Computational Efficiency of EP Calculation}
  \label{tab:dist_esr_eff}
  \renewcommand\arraystretch{1.3}
  \setlength{\tabcolsep}{2mm}{
  \begin{tabular}{ccccc}
    \toprule
    System &  \makecell{Number\\ of DERs} & \makecell{Time of\\FME (s)} & \makecell{Time of\\PVE (s)} & \makecell{Rate of \\ model reduction}\\
    \midrule
    DN-13 & 6 & 3.75 & 0.34 & 98.4\%\\
    DN-25 & 12 & 41.29 & 0.32  & 99.5\%\\
    DN-49 & 24 & \textgreater1200 & 0.37 & \textgreater99.9\% \\
    DN-241 & 120 & \textgreater1200 & 0.62 & \textgreater99.9\% \\
    DN-2401 & 1200 & \textgreater1200 & 2.37 & \textgreater99.9\%\\
    \bottomrule
    \end{tabular}
  }
\end{table}

Multiple IEEE-13 feeders are connected at the root node to create larger-scale test systems with 25, 49, 241, and 2401 nodes to examine the computation performance of the EP calculation. As shown in TABLE \ref{tab:dist_esr_eff}, the proposed PVE algorithm successfully obtains EP results of all five test systems within 3 seconds. In contrast, the conventional FME method takes more than 10 and 129 times of computational effort to calculate the EP for the 13-node and 25-node test systems, respectively. For test systems with more than 49 nodes, the FME fails to yield the EP results within 1200s. This result shows the superiority of the PVE-based EP calculation in terms of computational efficiency and scalability. As for the reduction of the model scale, the EP model of the distribution network is described by a group of 2-dimensional linear constraints, which reduces the model scale of the distribution network by more than 98.4\%.

\subsubsection{EP-based TDCOD}
\begin{table}[t]
  \centering
  \caption{Computational Efficiency of EP-based TDCOD}
  \label{tab:tdcd_eff}
  \renewcommand\arraystretch{1.3}
  \setlength{\tabcolsep}{2mm}{
  \begin{tabular}{ccccc}
      \toprule
      \multirow{2}{*}{\makecell{Num of\\DNs}} & \multirow{2}{*}{\makecell{Time of joint\\ optimization (s)}} & \multicolumn{3}{c}{Time of EP-based coordination (s)} \\
      \cmidrule{3-5}
       & &  \makecell{System \\ reduction} & \makecell{Coordinated \\ optimization} & Total \\
      \midrule
      $50$ & 0.26 & 0.31 & 0.08 & 0.39\\
      $100$ & 0.47 & 0.30 & 0.09 & 0.39\\
      $200$ & 0.68 & 0.31 & 0.12 & 0.43\\
      $400$ & 0.97 & 0.31 & 0.20 & 0.51\\
      \bottomrule
  \end{tabular}
  }
\end{table}

The IEEE-24 system is connected with different numbers of IEEE-13 distribution feeders to simulate the coordinated dispatch between transmission and distribution systems. The coordinated dispatch is solved by the joint optimization and the EP-based coordination method respectively. Optimization results from the two methods are verified to be identical for all the test cases. The computation time of the joint optimization and the EP-based coordination is summarized in TABLE \ref{tab:tdcd_eff}. The computation time of the EP-based coordination is dominated by the system reduction process. With the EP, the time of the coordinated optimization is reduced by more than 69.2\%, which may help to relieve the computation burden of the transmission system. The total computation time of the EP-based coordination is also reduced compared with the joint optimization in cases with more than 100 distribution networks. This is because the EP distributes the computation burden of solving a large-scale centralized optimization to each subsystem in parallel and thus, is more efficient for the coordinated management of numerous distribution networks. 


\section{Conclusion}
This two-part paper proposes the EP theory to make external equivalence of the subsystem and realizes the COD of power systems in a non-iterative fashion. To calculate the EP efficiently and accurately, Part II of this paper proposes a novel polyhedral projection algorithm termed the PVE, which characterizes the EP model by identifying its vertices and building the convex hull of vertices. The vertex identification process in the PVE algorithm is finely designed to give higher priority to vertices that are critical to the overall shape of the projection, which makes the approximation error decreases along the steepest path. The Hausdorff distance is employed to measure and control the calculation accuracy of the PVE algorithm, which can provide flexibility to balance the computation accuracy and computation effort of the EP calculation in practice. The EP theory and the PVE algorithm are applied to the non-iterative COD of multi-area systems and transmission-distribution systems. Case studies of different scales testify the superiority of the proposed PVE algorithm in terms of computational efficiency and scalability compared with the conventional FME algorithm. The EP is verified to reduce the subsystem model scale by more than 92.7\% for regional transmission systems and more than 98.4\% for distribution networks, which can alleviate the computation and communication burden in coordinated dispatch. The effectiveness of the EP-based non-iterative coordinated dispatch is also validated in comparison with the joint optimization.

The proposed EP theory and the non-iterative coordinated optimization framework are general and may find a broad spectrum of applications in addition to instances in this paper, e.g., the coordinated dispatch of the multi-energy network, the coupling of power and traffic networks, and the integration of user-side flexible resources. Applications of the PVE algorithm to other projection problems, e.g., flexibility region aggregation, loadability set \cite{ref:loadability1}, and the projection-based robust optimization \cite{ref:ro_prj}, are also worthy of future investigation. Incorporating the robust feasibility and chance constraints to deal with the uncertainty and incorporating inter-temporal constraints in the projection calculation should be future extensions of the proposed method.



\ifCLASSOPTIONcaptionsoff
  \newpage
\fi


\bibliographystyle{IEEEtran}
\bibliography{Reference}

\begin{thebibliography}{10}
\providecommand{\url}[1]{#1}
\csname url@samestyle\endcsname
\providecommand{\newblock}{\relax}
\providecommand{\bibinfo}[2]{#2}
\providecommand{\BIBentrySTDinterwordspacing}{\spaceskip=0pt\relax}
\providecommand{\BIBentryALTinterwordstretchfactor}{4}
\providecommand{\BIBentryALTinterwordspacing}{\spaceskip=\fontdimen2\font plus
\BIBentryALTinterwordstretchfactor\fontdimen3\font minus
  \fontdimen4\font\relax}
\providecommand{\BIBforeignlanguage}[2]{{%
\expandafter\ifx\csname l@#1\endcsname\relax
\typeout{** WARNING: IEEEtran.bst: No hyphenation pattern has been}%
\typeout{** loaded for the language `#1'. Using the pattern for}%
\typeout{** the default language instead.}%
\else
\language=\csname l@#1\endcsname
\fi
#2}}
\providecommand{\BIBdecl}{\relax}
\BIBdecl
\renewcommand{\BIBentryALTinterwordstretchfactor}{4}

\bibitem{ref:fme0}
D.~Bertsimas and J.~N. Tsitsiklis, \emph{Introduction to linear
  optimization}.\hskip 1em plus 0.5em minus 0.4em\relax Athena Scientific
  Belmont, MA, 1997, vol.~6.

\bibitem{ref:loadability1}
A.~Abiri-Jahromi and F.~Bouffard, ``On the loadability sets of power
  systems—{Part I}: Characterization,'' \emph{IEEE Trans. Power Systems},
  vol.~32, no.~1, pp. 137--145, 2016.

\bibitem{ref:ifme}
J.-L. Imbert, ``Fourier's elimination: Which to choose?'' in \emph{PPCP},
  vol.~1.\hskip 1em plus 0.5em minus 0.4em\relax Citeseer, 1993, pp. 117--129.

\bibitem{ref:loadability2}
A.~Abiri-Jahromi and F.~Bouffard, ``On the loadability sets of power
  systems—{Part II}: Minimal representations,'' \emph{IEEE Trans. Power
  Systems}, vol.~32, no.~1, pp. 146--156, 2016.

\bibitem{ref:block}
V.~Chandru, ``Variable elimination in linear constraints,'' \emph{The Computer
  Journal}, vol.~36, no.~5, pp. 463--472, 1993.

\bibitem{ref:mpp}
C.~N. Jones, E.~C. Kerrigan, and J.~M. Maciejowski, ``On polyhedral projection
  and parametric programming,'' \emph{Journal of Optimization Theory and
  Applications}, vol. 138, no.~2, pp. 207--220, 2008.

\bibitem{ref:box}
X.~Chen, E.~Dall’Anese, C.~Zhao, and N.~Li, ``Aggregate power flexibility in
  unbalanced distribution systems,'' \emph{IEEE Trans. Smart Grid}, vol.~11,
  no.~1, pp. 258--269, 2020.

\bibitem{ref:zono}
F.~L. Müller, J.~Szabó, O.~Sundström, and J.~Lygeros, ``Aggregation and
  disaggregation of energetic flexibility from distributed energy resources,''
  \emph{IEEE Trans. Smart Grid}, vol.~10, no.~2, pp. 1205--1214, 2019.

\bibitem{ref:ellipse}
E.~Polymeneas and S.~Meliopoulos, ``Aggregate modeling of distribution systems
  for multi-period {OPF},'' in \emph{2016 Power Systems Computation Conference
  (PSCC)}.\hskip 1em plus 0.5em minus 0.4em\relax IEEE, 2016, pp. 1--8.

\bibitem{ref:robust}
X.~Chen and N.~Li, ``Leveraging two-stage adaptive robust optimization for
  power flexibility aggregation,'' \emph{IEEE Trans. Smart Grid}, vol.~12,
  no.~5, pp. 3954--3965, 2021.

\bibitem{ref:pq1}
J.~Silva, J.~Sumaili, R.~J. Bessa, L.~Seca, M.~A. Matos, V.~Miranda,
  M.~Caujolle, B.~Goncer, and M.~Sebastian-Viana, ``Estimating the active and
  reactive power flexibility area at the {TSO-DSO} interface,'' \emph{IEEE
  Trans. Power Systems}, vol.~33, no.~5, pp. 4741--4750, 2018.

\bibitem{ref:jaccard}
M.~Levandowsky and D.~Winter, ``Distance between sets,'' \emph{Nature}, vol.
  234, no. 5323, pp. 34--35, 1971.

\bibitem{ref:md_dc}
B.~Stott, J.~Jardim, and O.~Alsa{\c{c}}, ``{DC} power flow revisited,''
  \emph{IEEE Trans. Power Systems}, vol.~24, no.~3, pp. 1290--1300, 2009.

\bibitem{ref:md_liudd}
D.~Liu, L.~Liu, H.~Cheng, S.~Zhang, and J.~Xin, ``An extended dc power flow
  model considering voltage magnitude,'' \emph{Journal of Modern Power Systems
  and Clean Energy}, vol.~9, no.~3, pp. 679--683, 2021.

\bibitem{ref:distflow2}
E.~Schweitzer, S.~Saha, A.~Scaglione, N.~G. Johnson, and D.~Arnold, ``Lossy
  {DistFlow} formulation for single and multiphase radial feeders,'' \emph{IEEE
  Trans. Power Systems}, vol.~35, no.~3, pp. 1758--1768, 2020.

\bibitem{ref:dist_wangyi}
Y.~Wang, N.~Zhang, H.~Li, J.~Yang, and C.~Kang, ``Linear three-phase power flow
  for unbalanced active distribution networks with {PV} nodes,'' \emph{CSEE
  Journal of Power and Energy Systems}, vol.~3, no.~3, pp. 321--324, 2017.

\bibitem{ref:md_hub}
Y.~Wang, N.~Zhang, C.~Kang, D.~S. Kirschen, J.~Yang, and Q.~Xia, ``Standardized
  matrix modeling of multiple energy systems,'' \emph{IEEE Trans. Smart Grid},
  vol.~10, no.~1, pp. 257--270, 2017.

\bibitem{ref:nonconv_price1}
M.~Garcia, H.~Nagarajan, and R.~Baldick, ``Generalized convex hull pricing for
  the {AC} optimal power flow problem,'' \emph{IEEE Trans. Control of Network
  Systems}, vol.~7, no.~3, pp. 1500--1510, 2020.

\bibitem{ref:equal_set}
C.~Jones, E.~C. Kerrigan, and J.~Maciejowski, ``Equality set projection: A new
  algorithm for the projection of polytopes in halfspace representation,''
  Cambridge University Engineering Dept, Tech. Rep., 2004.

\bibitem{ref:mw_theorem}
V.~Delos and D.~Teissandier, ``Minkowski sum of polytopes defined by their
  vertices,'' \emph{Journal of Applied Mathematics and Physics}, vol.~3, pp.
  62--67, 2015.

\bibitem{ref:hd}
D.~P. Huttenlocher, G.~A. Klanderman, and W.~J. Rucklidge, ``Comparing images
  using the {Hausdorff} distance,'' \emph{IEEE Trans. Pattern Analysis and
  Machine Intelligence}, vol.~15, no.~9, pp. 850--863, 1993.

\bibitem{ref:redundant1}
A.~J. Ardakani and F.~Bouffard, ``Identification of umbrella constraints in
  {DC}-based security-constrained optimal power flow,'' \emph{IEEE Trans. Power
  Systems}, vol.~28, no.~4, pp. 3924--3934, 2013.

\bibitem{ref:redundant2}
R.~Madani, J.~Lavaei, and R.~Baldick, ``Constraint screening for security
  analysis of power networks,'' \emph{IEEE Trans. Power Systems}, vol.~32,
  no.~3, pp. 1828--1838, 2017.

\bibitem{ref:qhull}
C.~B. Barber, D.~P. Dobkin, and H.~Huhdanpaa, ``The quickhull algorithm for
  convex hulls,'' \emph{ACM Trans. Mathematical Software}, vol.~22, no.~4, p.
  469–483, 1996.

\bibitem{ref:lin}
W.~Lin, Z.~Yang, J.~Yu, K.~Xie, X.~Wang, and W.~Li, ``Tie-line security region
  considering time coupling,'' \emph{IEEE Trans. Power Systems}, vol.~36,
  no.~2, pp. 1274--1284, 2021.

\bibitem{ref:node1}
P.~N. Biskas, D.~I. Chatzigiannis, and A.~G. Bakirtzis, ``European electricity
  market integration with mixed market designs—{Part I}: Formulation,''
  \emph{IEEE Trans. Power Systems}, vol.~29, no.~1, pp. 458--465, 2013.

\bibitem{ref:node2}
\BIBentryALTinterwordspacing
N.~Committee, ``{EUPHEMIA} public description,'' April 2019. [Online].
  Available:
  \url{https://www.nemo-committee.eu/assets/files/190410_Euphemia%20Public%20Description%20version%20NEMO%20Committee.pdf}
\BIBentrySTDinterwordspacing

\bibitem{ref:synthetic}
A.~B. Birchfield, T.~Xu, K.~M. Gegner, K.~S. Shetye, and T.~J. Overbye, ``Grid
  structural characteristics as validation criteria for synthetic networks,''
  \emph{IEEE Trans. Power Systems}, vol.~32, no.~4, pp. 3258--3265, 2017.

\bibitem{ref:tan_macd}
Z.~Tan, H.~Zhong, Q.~Xia, and C.~Kang, ``Non-iterative multi-area coordinated
  dispatch via condensed system representation,'' \emph{IEEE Trans. Power
  Systems}, vol.~36, no.~2, pp. 1594--1604, 2021.

\bibitem{ref:distflow}
M.~Baran and F.~Wu, ``Network reconfiguration in distribution systems for loss
  reduction and load balancing,'' \emph{IEEE Trans. Power Delivery}, vol.~4,
  no.~2, pp. 1401--1407, 1989.

\bibitem{ref:tan_rcc}
Z.~Tan, H.~Zhong, Q.~Xia, C.~Kang, X.~S. Wang, and H.~Tang, ``Estimating the
  robust {PQ} capability of a technical virtual power plant under
  uncertainties,'' \emph{IEEE Trans. Power Systems}, vol.~35, no.~6, pp.
  4285--4296, 2020.

\bibitem{ref:ro_prj}
J.~Zhen, D.~Den~Hertog, and M.~Sim, ``Adjustable robust optimization via
  {Fourier}--{Motzkin} elimination,'' \emph{Operations Research}, vol.~66,
  no.~4, pp. 1086--1100, 2018.

\end{thebibliography}

\end{document}